\documentclass{amsproc}

\usepackage{color}
\renewcommand{\dim}{\operatorname{dim}}
\newcommand{\Aut}{\operatorname{Aut}}

\newcommand{\End}{\operatorname{End}}

\newcommand{\Z}{{\mathbb Z}}
\newcommand{\M}{{\mathcal M}}
\newcommand{\Q}{{\mathbb Q}}
\newcommand{\G}{\mathcal{G}}
\newcommand{\E}{{\mathcal E}}

\newtheorem{theorem}{Theorem}[section]
\newtheorem{lemma}[theorem]{Lemma}
\newtheorem{proposition}[theorem]{Proposition}
\theoremstyle{definition}
\newtheorem{definition}[theorem]{Definition}

\newtheorem{corollary}[theorem]{Corollary}

\newtheorem{remark}[theorem]{Remark}

\usepackage{float}

\numberwithin{equation}{section}

\begin{document}

\title{Homotopically rigid Sullivan algebras and Their applications}

\author{Cristina Costoya}
\address{Departamento de Computaci\'on,
Universidade da Coru{\~n}a, Campus de Elvi{\~n}a, s/n, 15071 - A Coru{\~n}a, Spain.}
\email{cristina.costoya@udc.es}
\thanks{The first author was partially supported by Mi\-nis\-te\-rio de Econom\'ia y Competitividad (Spain), grant  MTM2016-79661-P,  and by Xunta de Galicia grant EM2013/16}%    Information for second author
\author{David M\'endez}
\address{Departamento de \'Algebra, Geometr{\'\i}a y Topolog{\'\i}a, Universidad de M{\'a}laga, E-29071-M{\'a}laga, Spain}
\email{davmenmar@uma.es}
\thanks{The second author was partially supported by Ministerio de Educaci\'on, Cultura y Deporte grant FPU14/05137, by FEDER-MEC grants MTM2013-41768-P and MTM2016-78647-P, by Ministerio de Econom\'ia y Competitividad (Spain) grant MTM2016-79661-P, and by Xunta de Galicia grant EM2013/16}
%    Information for  third author
\author{Antonio Viruel}
\address{Departamento de \'Algebra, Geometr{\'\i}a y Topolog{\'\i}a, Universidad de M{\'a}laga, E-29071-M{\'a}laga, Spain}
\email{viruel@uma.es}
\address{IRMP, Universit\'e catholique de Louvain, 2 Chemin du Cyclotron, B-1348 Louvain-la-Neuve, Belgium}
\email{pascal.lambrechts@uclouvain.be}
\address{University of Regina, Department of Mathematics and Statistics\\
College West, Regina, CANADA}%
\email{stanley@math.uregina.ca}
\thanks{The third author was partially supported by FEDER-MEC grants MTM2013-41768-P and MTM2016-78647-P, and by Xunta de Galicia grant EM2013/16}

%    General info
\date{January 1, 1994 and, in revised form, June 22, 1994.}

%\dedicatory{}


\begin{abstract}
In this paper we construct an infinite family of homotopically rigid spaces. These examples are then used as building blocks to forge highly connected rational spaces with prescribed finite group of self-homotopy equivalences.  They are also exploited to provide highly connected inflexible and strongly chiral manifolds.  \end{abstract}

\maketitle

\specialsection*{Introduction}
The group of self-homotopy equivalences of a space $X$, $\mathcal E ( X)$,  is rarely trivial.  An obvious example is $X= K ( \mathbb Z/2, n)$ since
$\mathcal E (K ( \mathbb Z/2, n)) \cong \Aut (\mathbb Z/2) \cong \ast$.   The first elaborated example with non trivial rational homology is constructed in \cite{ka1} by D. Kahn. He expressed in \cite{ka3} the belief that spaces with trivial group of self-homotopy equivalences, named  homotopically rigid spaces,  might play a role in some way of decomposing a space. Thus, the problem of determining $X$ for which $\mathcal E ( X)$ is trivial becomes of interest,  \cite[Problem 3]{ka3}.

A decade after,  Arkowitz and Lupton came across an example of a minimal Sullivan algebra (equivalently, a rational homotopy type of a space) $ M$  with $\mathcal E (M)$ trivial, \cite[Example 5.1]{AL2}. This, let us say homotopically rigid algebra, $M$ provides us with motivation and guidance: in our paper \cite{CV2},  we use $M$ as a building block to construct minimal Sullivan algebras, \cite[Definition 2.1]{CV2},  with the property that their group of self-homotopy equivalences are isomorphic to an arbitrary finite group $G$ previously fixed, \cite[Theorem 1.1]{CV2}.  In other words, every finite group $G$ is realizable through minimal Sullivan algebras that are built upon $M$;  this is a partial answer to Kahn's realizability problem and we refer to \cite{CV2} for more details. Observe that as an application, for the trivial group $G \cong \{e \}$, $M$ is then a building block for an infinite family of homotopically rigid minimal Sullivan algebras (equivalently, homotopically rigid rational spaces).
However, as interesting as it may seem, all the examples of the type above just mentioned are $7$-connected since $M$ is $7$-connected itself.

This paper came to light in our attempt to: (i) decide how rare homotopically rigid spaces are, since  very few examples are known in literature. In that sense, we construct an infinite family of such spaces as highly connected as wanted; and (ii) show that building blocks other than Arkowitz-Lupton example $M$ can be chosen in the resolution of Kahn's realizability problem in such a way that we are able to enlarge the family of examples in literature, all of them with low connectivity.

The paper is organized as follows. First, in Section \ref{sec:rigalg} we construct an infinite family of highly connected minimal Sullivan algebras (see Definition \ref{def: rigidalgebras})  and we prove that they are homotopically rigid  (see Theorem \ref{proposition: rigidalgebras}).  Second, in Section \ref{sec:realiz} we show that algebras from   Definition \ref{def: rigidalgebras} serve as building blocks for new highly connected minimal Sullivan algebras (see Definition \ref{def: m_nG})   with the property that their group of self-homotopy equivalences are isomorphic to any finite group $G$ that we fix (see Theorem \ref{teor:sullgraph} and Theorem \ref{cor:sullgraph}).  In other words, we are realizing the group $G$ through minimal Sullivan algebras that can be chosen as highly connected as desired.  Section  \ref{sec: inflex} consists of applications of the previous ones to differential geometry: we enlarge the class of inflexible manifolds existing in literature (see Theorem \ref{thminflexman}), and create new strongly chiral manifolds (see Corollary \ref{cor:chiral}, and Corollary \ref{cor:products}). We end this paper with an Appendix by Pascal Lambrechts and Don Stanley that allows us to guarantee the existence of highly connected inflexible manifolds.

\medskip
%%%%%%%%%%%%%%%%%%%%%%%%%%%%%%%%%%%%%%%%%%%%%%%%%%%%%%%%%%
\noindent {\bf Acknowledgments.} The authors would like to thank the referee for carefully reading our manuscript and for giving such valuable comments which helped improving the quality of the paper. 

 %%%%%%%%%%%%%%%%%%%%%%%%%%%%%%%%%%%%%%%%%%%%%%%%%%%%%%%%%%

\section{Highly connected homotopically rigid CDGA}\label{sec:rigalg}
In this section an infinite family of highly connected homotopically rigid minimal Sullivan algebras are constructed. They extend the family of examples in literature, all of them of low connectivity.

We refer to \cite{FHT2} for basic facts in rational homotopy theory. Only simply connected $\Q$-algebras of finite type are considered here. If $W$ is a graded rational vector space, we write $\Lambda W$ for the free commutative graded algebra on $W$. This is a symmetric algebra on $W^{\text{even}}$  tensored with an exterior algebra on $W^{\text{odd}}$. A Sullivan algebra is a nilpotent commutative differential graded algebra (CDGA) which is free as commutative graded algebra on a simply connected graded vector space $W$ of finite dimension in each degree. It is minimal if in addition $d(W) \subset \Lambda^{\geq 2}W$.  A Sullivan algebra is pure if  $d= 0$ on $W^{\text{even}}$ and $d (W^{\text{odd}}) \subset \Lambda (W^{\text{even}})$. We recall that the geometric realization functor of Sullivan \cite{Su} gives the equivalence of categories between the homotopy category of minimal Sullivan algebras, and the homotopy category of rational simply connected spaces of finite type. Henceforward,  by abuse of language, we sometimes refer to minimal Sullivan algebras as rational spaces.

\medskip

Recall that the homotopy rigid algebra constructed by Arkowitz and Lupton \cite[Example 5.1]{AL2} is
$$M=\big(\Lambda(x_1,x_2,y_1,y_2,y_3,z),d\big)$$ where dimensions and differential are
$$\begin{array}{ll}
\vert x_1\vert = 8,& dx_1=0\\
\vert x_2\vert = 10,& dx_2=0\\
\vert y_1\vert = 33,& dy_1=x_1^3x_2\\
\vert y_2\vert = 35,& dy_2=x_1^2x_2^2\\
\vert y_3\vert = 37,& dy_3=x_1x_2^3\\
\vert z\vert = 119,& dz=y_1y_2x_1^4x_2^2-y_1y_3x_1^5x_2+y_2y_3x_1^6+x_1^{15}+x_2^{12}.\\
\end{array}$$
Trying to obtain highly connected homotopy rigid algebras by re-scaling the degree of the generators in $M$ is hopeless: the differential in $M$ leads to a system of linear equations, involving the degree of the generators, whose unique solution is the one given by $M$.

%Therefore, in order to obtain homotopy rigid algebras inspired on the Arkowitz and Lupton example, the differential has to be modified.

We now introduce the main building blocks of this paper.
\begin{definition}\label{def: rigidalgebras}
For any even integer, $k > 4$, we define the minimal Sullivan algebra $${M}_k = \big(\Lambda(x_1,x_2,y_1,y_2,y_3,z),d\big)$$ where
\[\begin{array}{ll}
|x_1| = 5k-2,  &				dx_1=0, \\
|x_2| = 6k-2, &	 				dx_2=0, \\
|y_1| = 21k-9, &				dy_1 = x_1^3 x_2, \\
|y_2| = 22k-9, &				dy_2 = x_1^2 x_2^2, \\
|y_3| = 23k-9, & 				dy_3 = x_1 x_2^3, \\
|z| = 15k^2 - 11k + 1, &			dz = x_1^{3k-12}(x_1^2y_2y_3 - x_1 x_2 y_1 y_3 + x_2^2 y_1 y_2) + x_1^{\frac{6k-2}{2}} + x_2^{\frac{5k-2}{2}}.
\end{array}\]
\end{definition}
\begin{remark}\label{rem:parity}
Observe that since $k$ is even,  only $x_1, x_2$ have even degree; the rest of generators have odd degree.  This is a technicality which proves to be crucial in the proof of  Theorem \ref{proposition: rigidalgebras}, the main result in this section.  \end{remark}

For the sake of clarity, the proof of Theorem \ref{proposition: rigidalgebras} is split into three lemmas. Recall that the algebras $M_k$ are introduced in Definition \ref{def: rigidalgebras}.

\subsection*{Lemmas on generators of $M_k$}
 The first two lemmas measure how the degrees of the generators are, in some sense, far apart.
\begin{lemma}\label{lemma: isolated}
Generators of ${M}_k$ satisfy the following inequalities:
\[|x_1| < |x_2| < |y_1| < |y_2| < |y_3| < |x_1 y_1| < |x_2 y_3| < |z|.\]
\end{lemma}

\begin{proof}
All the inequalities are straightforward except for the last one, which follows from the fact that $x_2 y_3$ is a factor in a term of $dz$, and ${M}_k$ has no generator in degree $1$.
\if
is equivalent to prove that
$0 < 15k^2 - 40k + 12$. This happens whenever
$k \not \in \left[\frac{20-2\sqrt{55}}{15}, \frac{20 + 2\sqrt{55}}{15}\right]$,
which is our case since $\frac{20 + 2\sqrt{55}}{15}<4<k$.
\fi
\end{proof}

\begin{lemma}\label{lemma: divalg}
None of the following integers are divisible by $|x_1| $ or by $|x_2| $:

\begin{table}[H]
\caption{}\label{eqtable1}
\renewcommand\arraystretch{1.5}
\noindent\[
\begin{array}{|l|l|l|}
\hline
|z| - |y_1| &|z| - |y_2|&|z| - |y_3|\\
\hline
|z| - |y_1| - |x_1|-|x_2|&|z| - |y_2| - |x_1|-|x_2|&|z| - |y_3|- |x_1|-|x_2| \\
\hline
\end{array}
\]
\end{table}

\end{lemma}

\begin{proof}
We first show that they are not  divisible by $|x_1|$, which is an even number. Let us write $ |x_1|= 5k-2 = 2n$. Then, for $k = \frac{2n+2}{5}$, {\sc{Table} }\ref{eqtable1} becomes:

\begin{table}[H]
\renewcommand\arraystretch{1.5}
\noindent\[
\begin{array}{|l|l|l|}
\hline
\frac{1}{5}(12n^2 - 40n - 2)& \frac{1}{5}(12n^2 - 42n - 4)& \frac{1}{5}(12n^2 - 44n - 6)\\
\hline
\frac{1}{5}(12n^2 - 62n - 4)& \frac{1}{5}(12n^2 - 64n - 6)& \frac{1}{5}(12n^2 - 66n - 8) \\
\hline
\end{array}
\]
\end{table}
Let $f$ be any of the polynomials in the table above. If $f$ is divisible by $|x_1|=2n$, then $5f$ is so. Every monomial in $5f$ is divisible by $2n$ but the independent term, which is at most $8$ thus not divisible by $2n>18$. Therefore neither of $5f$ and $f$ is divisible by $2n$.

In a similar way we prove that they are not divisible by $|x_2|$, which is also an even number. Let us write $|x_2|=6k-2 =2n$. For  $k = \frac{n+1}{3}$ {\sc{Table}} \ref{eqtable1} becomes:
\begin{table}[H]
\renewcommand\arraystretch{1.5}
\noindent\[
\begin{array}{|l|l|l|}
\hline
\frac{1}{3}(5n^2 - 22n + 3) &	 \frac{1}{3}(5n^2 - 23n + 2)&		 \frac{1}{3}(5n^2 - 24n + 1)\\
 \hline
 \frac{1}{3}(5n^2 - 33n + 4) &  \frac{1}{3}(5n^2 - 34n + 3) &  	  \frac{1}{3}(5n^2 - 35n + 2) \\
 \hline
\end{array}%
\]
\end{table}
Again, let $f$ be any of the polynomials in the table above. If $f$ is divisible by $|x_2|=2n>22$, then $f$ and $3f$ are divisible by $n$. Every monomial in $3f$ is divisible by $n$ but the independent term, which is at most $4$ thus  not divisible by $n>11$. Therefore neither of $3f$ and $f$ is divisible by $n$, and $f$ is not divisible by $2n$.
\end{proof}

The third and final lemma describes some specific elements in $M_k$. By $M_k^s$ we denote the elements of degree $s$.
\begin{lemma}\label{lem: dmzinflex}
Let $A_i \in \mathbb{Q}[x_1,x_2]$ be such that $y_i A_i \in {M}_k^{|z|}$. Then $A_i$ is a multiple of $(x_1 x_2)^2$, for $i=1, 2, 3$. Moreover, if $\sum_{i=1}^3 y_i A_i$ is a cocycle, then it is also a coboundary.
\end{lemma}

\begin{proof}
For $y_i A_i \in {M}_k^{|z|}$, then $ |A_i| = |z| - |y_i| $ which, by Lemma \ref{lemma: divalg},  is not divisible by $|x_1|$ or by $|x_2|$.  In this situation, $A_i$ has not monomials that are pure powers of $x_1$ or $x_2$ and, as a consequence,  $A_i$ is a multiple of $x_1 x_2$.  Resulting from that $A_i = x_1 x_2 B_i$ for some $B_i \in \mathbb{Q}[x_1, x_2]$ with $|B_i| = |z| - |y_i| - |x_1| - |x_2|$ which, again by Lemma \ref{lemma: divalg}, is not divisible  by $|x_1|$ or by $|x_2|$. The same argument as above shows that $B_i$ is a multiple of $x_1 x_2$, so $B_i = x_1 x_2 C_i$ for some $C_i \in \mathbb{Q}[x_1, x_2]$. Then, $A_i = (x_1 x_2)^2 C_i$.

Moreover, if $\Sigma_{i=1}^3 y_i A_i $ is a cocycle,
$d(\Sigma_{i=1}^3 y_i A_i) = x_1^3 x_2 A_1 + x_1^2 x_2^2 A_2 + x_1 x_2^3 A_3 = (x_1 x_2)^3 (x_1^2 C_1 + x_1 x_2 C_2 + x_2^2 C_3)= 0$
leading to $x_1^2 C_1 + x_1 x_2 C_2 + x_2^2 C_3 = 0$. This equality forces $C_3$ and $C_1$ to be multiples of $x_1$ and $x_2$ respectively. Then,
$C_1 = x_2 \bar{C}_1$ and $C_3 = x_1 \bar{C}_3$
for some $\bar{C}_1, \bar{C}_3 \in \mathbb{Q}[x_1,x_2]$ and the equality above becomes $x_1 x_2 (x_1 \bar{C}_1 + C_2 + x_2\bar{C}_3)=0. $
Therefore $C_2 = -x_1 \bar{C}_1 - x_2 \bar{C}_3$ and
\begin{align*}
 y_1 A_1 + y_2 A_2 + y_3 A_3 &=y_1 (x_1 x_2)^2 C_1 + y_2(x_1 x_2)^2  C_2 + y_3 (x_1 x_2) ^2 C_3 \\
& = y_1(x_1 x_2)^2 x_2 \bar{C}_1 + y_2(x_1 x_2)^2  (-x_1\bar{C}_1 - x_2 \bar{C}_3)+ y_3(x_1 x_2)^2 x_1 \bar{C}_3 \\
& = \bar{C}_1 x_2 (y_1 x_1^2 x_2^2 - y_2 x_1^3 x_2) + \bar{C}_3 x_1(y_3 x_1^2 x_2^2 - y_2 x_1 x_2^3) \\
& = \bar{C}_1 x_2 d(y_2 y_1) + \bar{C}_3 x_1 d(y_2 y_3)\\
& = d(\bar{C}_1 x_2 y_2 y_1 + \bar{C}_3 x_1 y_2 y_3),
\end{align*}
which exhibits $\sum_{i=1}^3 y_i A_i$ as a coboundary.
\end{proof}

\subsection*{Main theorem in Section  \ref{sec:rigalg}}
We have all the ingredients to present and prove the main result in this section. Recall that $M_k$, for any even $k >4$, are the $(5k-3)$-connected algebras introduced in Definition \ref{def: rigidalgebras}. The following result directly implies that $M_k$ are homotopically rigid algebras:
\begin{theorem}\label{proposition: rigidalgebras}
The monoid of the homotopy classes of self-maps of $M_k$ has only two elements: the zero map and the identity.
\end{theorem}
\begin{proof}
Let $f\in\operatorname{End}({M}_k)$. By degree reasons, following Lemma~\ref{lemma: isolated} we know that
$f(x_i) = a_i x_i, \,  f(y_j) = b_j y_j, $  and
$f(z) = cz + \Sigma_{i=1}^3 A_i y_i + B y_1 y_2 y_3$,  $a_i,    b_j,  c \in \Q$,  $A_1, A_2, A_3, B \in \mathbb{Q}[x_1,x_2]$.
We point out here that since $f(z)$ is of odd degree, no summand from $\Lambda (x_1, x_2 ) \otimes \Lambda^2(y_1, y_2, y_3)$ can  be part of it. 

While $df(y_i) = f(dy_i)$ allow us to immediately obtain the following relations
\begin{equation}\label{eq: factoresxiyi}
b_1 = a_1^3 a_2, \quad b_2 = a_1^2 a_2^2 \quad \text{and} \quad b_3 = a_1 a_2^3,
\end{equation}
comparing $f(dz)= df(z)$ is more demanding. On one hand we have that
\begin{multline}\label{eq: fdzalgrig}
f(dz) = a_1^{3k-12} x_1^{3k-12}(a_1^2 b_2 b_3 x_1^2 y_2 y_3 - a_1 a_2 b_1 b_3 x_1 x_2 y_1 y_3 + a_2^2 b_1 b_2 x_2^2 y_1 y_2) \\
+ a_1^{\frac{6k-2}{2}}x_1^{\frac{6k-2}{2}} + a_2^{\frac{5k-2}{2}}x_2^{\frac{5k-2}{2}},
\end{multline}
and on the other hand
\begin{multline}\label{eq: dfzalgrig}
df(z) = c\left[x_1^{3k-12}(x_1^2y_2y_3 - x_1 x_2 y_1 y_3 + x_2^2 y_1 y_2) + x_1^{\frac{6k-2}{2}} + x_2^{\frac{5k-2}{2}}\right] \\
+ x_1^3 x_2 A_1 + x_1^2 x_2^2 A_2 + x_1 x_2^3 A_3 + B(x_1^3 x_2 y_2 y_3 - x_1^2 x_2^2 y_1 y_3 + x_1 x_2^3 y_1 y_2).
\end{multline}
Note that none of the summands in \eqref{eq: fdzalgrig} is a multiple of $x_2 y_2 y_3$, whereas $Bx_1^3x_2 y_2 y_3$ appears as a summand in \eqref{eq: dfzalgrig}. This forces $B=0$. Also, if we compare monomials of the form $x_1 x_2 P$, with $P\in \mathbb{Q}[x_1,x_2]$,  in \eqref{eq: fdzalgrig} and \eqref{eq: dfzalgrig}, we are led to the conclusion that $x_1^3 x_2 A_1 + x_1^2 x_2^2 A_2 + x_1 x_2^3 A_3 = 0$. This implies that $\sum_{i=1}^3 y_i A_i$ is a cocycle and, by Lemma \ref{lem: dmzinflex},  a coboundary. Following the strategy of comparing the rest of the coefficients in both equations and considering  \eqref{eq: factoresxiyi} we get
\begin{equation}\label{eq: factorescaibi}
\left\{\begin{array}{l}
c = a_1^{3k-10}  b_2 b_3 = a_1^{3k-7} a_2^5, \\
c = a_1^{3k-11} a_2 b_1 b_3 = a_1^{3k-7} a_2^5, \\
c = a_1^{3k-12} a_2^2 b_1 b_2 = a_1^{3k-7} a_2^5, \\
c = a_1^{3k-1}, \\
c = a_2^{\frac{5k-2}{2}}.
\end{array}\right.
\end{equation}
When $a_1 = 0$, it is easy to obtain that $b_1 = b_2 = b_3 = a_1 = a_2 = c = 0$.  In other case, as $a_1^{3k-1} = a_2^{\frac{5k-2}{2}}$ and also $a_1^{3k-1} =  a_1^{3k-7} a_2^5$,  we get that $a_1^6 = a_2^5$, or equivalently, $a_1^3 = a_2^{\frac{5}{2}}$. Combining both identities, we have that $a_1^{3k-1} = a_1^{3k-2} a_2$.
Therefore  $a_1 = a_2$, and since $a_1^6 = a_2^5$, we deduce that $a_1 = a_2 = 1$.   So, when $a_1 \neq 0$,  we infer that  $a_1 = a_2 = b_1 = b_2 = b_3 = c = 1$.

Summarizing all the steps, we have proved that for $f\in \operatorname{End}({M}_k)$, there exists $s=0,1,$ such that
$f(x_i) =  s x_i, \,f(y_j) = s y_j,  \, f(z) = s z + d(m_z),\, m_z \in M^{|z|-1}.$   From there,  it is only necessary to observe that $f (z) = sz + d(m_z)$ is homotopically equivalent to $\tilde f (z) = s z$ and we have succeeded in proving that the monoid of the homotopy classes of self-maps  of $M_k$ has only two elements: the zero map (for $s=0$), and the identity  (for $s=1$). \end{proof}

\section{Realizing finite groups by highly connected CDGA}\label{sec:realiz}
The purpose of this section is to prove that algebras from Definition \ref{def: rigidalgebras} serve as building blocks in the resolution of Kahn's realizability problem.  Recall that Kahn's problem asks, for an arbitrary group $G$,  if  $G \cong \mathcal E (X)$ for some space $X$; if so, $G$ is said to be realized by $X$.  We argue as in \cite[\S 2]{CV2}: for any finite group $G$, by Frucht's theorem \cite{Frucht1},  there exists a finite, connected and simple graph $\mathcal G = (V (\mathcal G ), E (\mathcal G))$ whose automorphism group, $\rm{Aut} (\mathcal G)$, is isomorphic to $G$. The landscape changes now from groups to graphs, and our approach consists on realizing the group $\Aut(\mathcal G)$ by the following minimal Sullivan algebra:
\begin{definition}\label{def: m_nG}
Let $n\ge 1$ be an integer, and $\G$ be a finite, connected and simple graph with more than one vertex. We define $${M}_n(\mathcal{G})=\big(\Lambda (x_1,x_2,y_1,y_2,y_3,z)\otimes \Lambda\big(x_v,z_v\mid v\in V(\mathcal{G})\big),d\big)$$ where
\[\begin{array}{ll}
|x_1| = 30n+18,  &				dx_1=0, \\
|x_2| = 36n+22, & 				dx_2=0, \\
|y_1| = 126n+75, &				dy_1 = x_1^3 x_2, \\
|y_2| = 132n+79, &				dy_2 = x_1^2 x_2^2, \\
|y_3| = 138n+83, & 				dy_3 = x_1 x_2^3, \\
|x_v| = 180n^2 +218n+66, & 		dx_v = 0, \\
|z| = 540n^2 + 654n + 197, &		dz = x_1^{18n}(x_1^2y_2y_3 - x_1 x_2 y_1 y_3 + x_2^2 y_1 y_2), \\
& 							\hspace{16pt}+  x_1^{18n+11} + x_2^{15n+9} \\
|z_v| = 540n^2 + 654n + 197, &	dz_v = x_v^3 + \sum_{(v,w)\in E(\G)} x_v x_w x_2^{5n+3}.
\end{array}\]
\end{definition}

\begin{remark}\label{rmk:como es M(G)}
Observe that $M_n (\mathcal G)$ is described as a tensorial product where the left factor is the minimal Sullivan algebra $M_{6n+4}$ from Definition \ref{def: rigidalgebras}, and the right factor codifies $\mathcal G$. Indeed, for every vertex $v\in V({\mathcal G})$, $v$ is represented by the generator $x_v$, while $dz_v$ codifies the neighborhood of $v$, that is, the edges in $E (\mathcal G)$ that are incident to the vertex $v$. Finally, notice that we are considering simple non oriented graphs, i.e.\ $E (\mathcal G)$ consists of non ordered pairs of vertices, and therefore the pairs $(v,w)$ and $(w,v)$ represent the same element in $E ({\mathcal G})$. 
\end{remark}

%Recall that the algebras $M_n (\mathcal G), \,n \geq 1,$ are introduced in Definition  \ref{def: m_nG} and the letters $x_1, x_2, y_1, y_2, y_3,  x_v, z, z_v$ refer to its generators.

\subsection*{Lemmas on generators of  $M_n (\mathcal G)$}\label{sub:gen}
We prove now three technical lemmas that are needed for the main result in this section, Theorem \ref{teor:sullgraph}. 	

\begin{lemma}\label{lemma:divalg2}
Let us consider
$r_1= |z| - |y_1| $, $r_2= |z| - |y_2|$ and $r_3= |z| - |y_3|$.  The following integers are not divisible by either $|x_1|$ or $|x_2|$:
\begin{table}[ht]
\caption{}\label{eqtable2}
\renewcommand\arraystretch{1.5}
\noindent\[
\begin{array}{|l|l|l|}
\hline
r_1&r_2&r_3\\
\hline
r_1 - |x_1|-|x_2|&r_2 - |x_1|-|x_2|&r_3- |x_1|-|x_2| \\
\hline
r_1-|x_v|&r_2-|x_v|&r_3-|x_v|\\
\hline
r_1-|x_v|- |x_1|-|x_2|	& r_2-|x_v|- |x_1|-|x_2|	& r_3-|x_v|- |x_1|-|x_2|\\
\hline
r_1-2|x_v|  	& r_2-2|x_v|  	&r_3-2|x_v|\\
\hline
r_1-2|x_v|-|x_1|-|x_2|	 & r_2-2|x_v|-|x_1|-|x_2|	&r_3-2|x_v|-|x_1|-|x_2| \\
\hline
\end{array}
\]
\end{table}
\end{lemma}

\begin{proof}

Nothing is to prove for the two top rows, it follows from Lemma \ref{lemma: divalg} by taking $k=6n+4$. For the rest, we start by showing that they are not divisible by $|x_1|=30n+18$, which is an even number. Let us write $|x_1|=2m$, and $n=\frac{m-9}{15}$.  Then, the four bottom rows become:
	\begin{table}[H]
\renewcommand\arraystretch{1.5}
\noindent\[
\begin{array}{|l|l|l|}
\hline
	\frac{1}{15}(24m^2-122m-6)		& \frac{1}{15}(24m^2-128m-12)	& \frac{1}{15}(24m^2 - 134m - 18)\\
	\hline
	\frac{1}{15}(24m^2 -188m -12)	& \frac{1}{15}(24m^2 - 194m - 18)	& \frac{1}{15}(24m^2 - 200m - 24) \\
	\hline	\frac{1}{15}(12m^2 - 124m - 6)	& \frac{1}{15}(12m^2-130m-12)	& \frac{1}{15}(12m^2 - 136m - 18)\\
	\hline
		\frac{1}{15}(12m^2 -190m - 12)	& \frac{1}{15}(12m^2 - 196m - 18)	& \frac{1}{15}(12m^2 -202m -24)\\
\hline
\end{array}
\]
\end{table}
Let $f$ be any of the polynomials in the table above: if $f$ is divisible by $|x_1|=2m$, then $15f$ is so. Every monomial in $15f$ is divisible by $2m$ but the independent term, which is at most $24$ thus not divisible by $|x_1|=2m=30n+18\ge 48$. Therefore neither of $15f$ nor $f$ is divisible by $|x_1|$.

Similar arguments show that none of the integers from the table is divisible by $|x_2|=36n+22$, which is also an even number. Let us write $|x_2| = 2m$,  and $n=\frac{m-11}{18}$. Then, the four bottom rows in {\sc{Table}} \ref{eqtable2} become:

		\begin{table}[H]
\renewcommand\arraystretch{1.5}
\noindent\[
\begin{array}{|l|l|l|}
\hline
	\frac{1}{9}(10m^2 - 65m + 9)	& \frac{1}{9}(10m^2 - 68m + 6)	& \frac{1}{9}(10m^2 - 71m + 3) \\
	\hline
	\frac{1}{9}(10m^2 - 98m + 12)	& \frac{1}{9}(10m^2 - 101m + 9)	& \frac{1}{9}(10m^2 - 104m + 6)\\
	\hline
	\frac{1}{9}(5m^2 - 64m + 9)		& \frac{1}{9}(5m^2 - 67m + 6)	& \frac{1}{9}(5m^2 - 70m + 3) \\
	\hline
	\frac{1}{9}(5m^2 - 97m + 12)	& \frac{1}{9}(5m^2 - 100m + 9)	& \frac{1}{9}(5m^2 - 103m + 6)\\
\hline
\end{array}
\]
\end{table}
Again, let $f$ be any of the polynomials in the table above: if $f$ is divisible by $|x_2|=2m$, then $f$ and $9f$ are divisible by $m$. Every monomial in $9f$ is divisible by $m$ but the independent term, which is at most $12$ thus  not divisible by $m=18n+11 \ge 29$. Hence neither of $9f$ nor $f$ is divisible by $m$.
\end{proof}

The description of some specific elements in the algebra ${M}_n(\mathcal{G})$ is also needed:

\begin{lemma}\label{prop:dmz}
Let $i=1,2,3$ and $A_i\in\mathbb{Q}[x_1,x_2]$. If $y_i A_i \in \big({M}_n(\mathcal{G})\big)^m$,  for $m= |z|$, $|z|-|x_v|$, or $|z|-2|x_v|$, then $A_i=(x_1x_2)^2 \bar{A}_i$, where $\bar{A}_i\in\mathbb{Q}[x_1,x_2]$. Moreover, if $\sum_{i=1}^3 y_i A_i$ is a cocycle, then it is a coboundary.
\end{lemma}
\begin{proof}
If $y_iA_i \in \big({M}_n (\mathcal{G})\big)^{m}$, then $|A_i| = m-|y_i|$. When $m=|z|$, $|z|-|x_v|$,  or $|z|-2|x_v|$, $|A_i|$ is respectively the $i$-th term of the first, third or fifth row in {\sc{Table}}  \ref{eqtable2}.  By Lemma \ref{lemma:divalg2},  as $|A_i|$ is not divisible by neither $|x_1|$ nor $|x_2|$, we infer that $A_i$ is not a pure power of $x_1$ nor $x_2$, so $A_i = x_1x_2 B_i$, where $B_i\in\mathbb{Q}[x_1,x_2]$. Now,  for $m=|z|$, $|z|-|x_v|$,  or $|z|-2|x_v|$,  $|B_i|$ is, respectively, the $i$-th term of the second, fourth or sixth row in {\sc{Table}}  \ref{eqtable2}. Hence $|B_i|$ is not divisible by neither $|x_1|$ nor $|x_2|$,  and the same arguments allow us to write $B_i=x_1x_2 \bar{A}_i$. Thus, we finally obtain that $A_i = (x_1 x_2)^2 \bar{A}_i$, where $\bar{A}_i\in \mathbb{Q}[x_1, x_2]$.
	
Similar arguments as in the proof of Lemma \ref{lem: dmzinflex} show that if $\sum_{i=1}^3 y_iA_i$ is a cocycle then it is a coboundary; details are omitted. \end{proof}

The last technical result that we need is the following.

\begin{lemma}\label{lem:dioph}
There is no pair of positive integers $(\alpha,\beta)$ such that
	\begin{equation} \label{eq:dioph} m = \alpha |x_1| + \beta |x_2|,\end{equation}
	for $m = |x_v|$. Moreover,
	there is no pair of non-negative integers $(\alpha,\beta)$ verifying \eqref{eq:dioph}  for $m \in \{ |z|-|y_1y_2y_3|,  |z|-|x_v y_1 y_2 y_3|,  |z|-|x_v^2 y_1 y_2 y_3|\}$.
\end{lemma}
\begin{proof}
To obtain the general solution of the diophantine  equation \eqref{eq:dioph}, first notice that
\begin{align*}
-6 |x_1|+ 5|x_2|&=-6(30n+18)+5(36n+22),\\
&=2,
\end{align*}
so the greatest common divisor of $|x_1|$ and $|x_2|$ is  $2$.   Thus it suffices to compute a particular solution $(\alpha,\beta)$, and express the general one as $$\big(\alpha + k\frac{|x_2|}{2}, \beta - k\frac{|x_1|}{2}\big)=\big(\alpha + k(18n+11), \beta - k(15n+9)\big), \; k\in\mathbb{Z}.$$

For $m = |x_v|= 180n^2 + 218n + 66$,  the pair $(0,5n+3)$ is a particular solution for \eqref{eq:dioph} and the general solution is $\big( k(18n+11), 5n+3 - k(15n+9)\big)$, $k\in\mathbb{Z}$.  It is clear that for $k>0$, the second component is negative, and for $k\le 0$, the first component is non positive. Thus there is no solution where both integers are positive.

For the rest of cases we just write the general solution below here. It is straightforward to check that at least one of the two components of the pair is negative for all $k \in\mathbb{Z}$:
\begin{table}[H]
\renewcommand\arraystretch{1.5}
\noindent\[
\begin{array}{|c|c|}
\hline
m & \text{general solution}\\
\hline
 |z|-|y_1y_2y_3|&\big(-12+k(18n+11),15n+8-k(15n+9) \big)\\
\hline
|z|-|x_v y_1 y_2 y_3|&\big(-12+k(18n+11), 10n+5-k(15n+9)\big)\\
\hline
|z|-|x_v^2 y_1 y_2 y_3|& \big(-12 + k(18n+11),  5n+2 - k(15n+9)\big)\\
\hline
\end{array}
\]
\end{table}
\noindent
\end{proof}

\subsection*{Sullivan algebras $M_n(\G)$ are elliptic}
First, we recall that a Sullivan algebra $(\Lambda W, d)$ is said to be elliptic when both $W$  and $H^\ast (\Lambda W)$ are finite-dimensional. The cohomology of an elliptic Sullivan algebra verifies Poincar\'e duality \cite{Hal}. One can compute the degree of its fundamental class (a fundamental class of a Poincar\'e duality algebra $H=\Sigma_{i=0}^m H^i$, is a generator of $H^m$, $m$ is said to be the formal dimension of the algebra) by the formula \cite[Theorem 32.6]{FHT2}:
	\begin{equation}\label{dimensiontop}
	\textstyle{ \sum_{i=1}^p}(\deg y_i) - \sum_{j=1}^q (\deg x_j -1)
	\end{equation}
	where $\deg y_i$ are the degrees of the elements of a basis of $W^{\text{odd}}$ and $\deg x_j$ of a basis of $W^{\text{even}}$.

The existence of a fundamental class for elliptic Sullivan algebras gives rise to a notion of degree for self maps.

\begin{definition}\label{def:degree}
Let $M$ be an elliptic Sullivan algebra with fundamental class $c_M\in H^m(M)$. Given $f\in \End (M)$ we say that the degree of $f$ is $k\in\mathbb{Q}$, $\deg(f)=k$, if $f^\ast(c_M)=kc_M$.
\end{definition}

The following result is inspired in classical works by Hopf \cite{Hopf1, Hopf2}:
%, shows that self-homotopy equivalences of elliptic Sullivan algebras are characterized by their degree:
\begin{proposition}\label{prop:hopf}
Let $M$ be an elliptic Sullivan algebra and $f\in [M,M]$. Then $f$ is a self-homotopy equivalence if and only if $\deg(f)\ne 0$.
\end{proposition}
\begin{proof}
It is clear that a self homotopy equivalence has a non-zero degree since $\deg(Id)=1$ and the degree is a multiplicative invariant. Suppose now that $\deg(f)\ne 0$. Then, we are going to prove that $H^*(f)$ is injective. Indeed,  assume that here exist a non trivial class $\alpha\in H^p(M)$ such that $f^\ast(\alpha)=0$. Now, since $H^*(M)$ is a Poincar\'e duality algebra, consider $c_M\in H^m(M)$ the fundamental class and the (non-degenerate) bilinear form $$\langle\,,\rangle\colon H^p(M)\otimes H^{m-p}(M)\to \mathbb{Q}$$ given by $xy=\langle x, y\rangle c_M$. This bilinear form induces an isomorphism $$\theta\colon H^p(M)\to H^{m-p}(M)^\sharp$$ given by $\theta(x)(y)=\langle x, y\rangle$. Therefore, since $\alpha\ne 0$, there exists $\beta\in H^{m-p}(M)$ such that $\alpha\beta=c_M$. Then $0=f^\ast(\alpha)f(\beta)=f^\ast(\alpha\beta)=f^\ast(c_M)=\deg(f)c_M$
which implies $0=\deg(f)$, contradicting our initial assumption. Hence  $H^*(f)$ is injective and,  as we are considering finite type algebras, $f^\ast$ is a quasi-isomorphism, which means that  $f$ is a self-homotopy equivalence.
\end{proof}

Finally, we prove the following lemma.
\begin{lemma}\label{lemma: ellipticminimal}
The algebra ${M}_n(\mathcal{G})$ is an elliptic Sullivan algebra
	of formal dimension $540n^2 + 984n + 396 + |V(\G)|(360n^2 + 436n + 132)$.
\end{lemma}
\begin{proof}
	We have to prove the cohomology of ${M}_n(\mathcal{G})$ is finite-dimensional. Instead, we prove the equivalent condition that the pure Sullivan algebra associated to ${M}_n(\mathcal{G})$,\cite[Proposition 32.4]{FHT2}, is also elliptic. This pure algebra, denoted by $(M_n^\varsigma(\mathcal{G}), d_\varsigma)$, has the same generators as  ${M}_n(\mathcal{G})$ and differentials
	\[\begin{array}{ll}
	d_\varsigma (x_1)=0, & d_\varsigma(x_2) = 0, \\
	d_\varsigma(y_1) = x_1^3 x_2, & d_\varsigma(y_2) = x_1^2 x_2^2, \\
	d_\varsigma(y_3) = x_1 x_2^3, & d_\varsigma(z) = x_1^{18n+11} + x_2^{15n+9}, \\
	d_\varsigma(x_v) = 0, \ v\in V(\G), & d_\varsigma(z_v) = x_v^3 + \Sigma_{(v,w)\in E(\G)} x_v x_w x_2^{5n+3}, \ v\in V(\G).
	\end{array}\]
	The cohomology of ${M}^\varsigma_n(\mathcal{G})$ is finite-dimensional because
	\[d_\varsigma(z x_1 - x_2^{15n+6} y_3) = x_1^{18n+12} \quad \text{and} \quad d_\varsigma(z x_2 - x_1^{18n+8}y_1) = x_2^{15n+10},\] and the cohomology class
	\[[x_v^3]^4 = \left[- \sum_{(v,w)\in E(\G)}  x_v x_w x_2^{5n+3}\right]^4=0.\]
	Hence ${M}_n(\mathcal{G})$ is elliptic and its formal dimension is obtained by \eqref{dimensiontop}.
\end{proof}

\subsection*{Main theorem in Section \ref{sec:realiz}}
We present the main result in this section. For $M_n(\G), \, n \geq 1,$ the minimal Sullivan algebra from Definition \ref{def: m_nG}, we prove:

\begin{theorem}\label{teor:sullgraph}
For any finite simple graph $\mathcal G$, the $(30n+17)$-connected algebra  $M_n ( \G)$ realizes $\Aut (\mathcal G)$, i.e.,  $$\mathcal E (M_n ( \G)) \cong \Aut (\mathcal G).$$
\end{theorem}

Before getting into the details of the proof, one immediate consequence:

\begin{theorem}\label{cor:sullgraph}
Any finite group $G$ can be realized by infinitely many highly connected rational spaces.
 \end{theorem}
 \begin{proof}
 In  \cite{Frucht1}, Frucht proves that every finite group $G$ is the group of automorphisms of a finite simple graph $\mathcal G$. It suffices then to apply Theorem \ref{teor:sullgraph}.
 \end{proof}

For the proof of Theorem \ref{teor:sullgraph}, two more lemmas are needed.  Recall from Remark \ref{rmk:como es M(G)} that algebras ${M}_n(\mathcal{G})$ are constructed to codify the graph $\mathcal{G}$ in terms of vertices (generators $x_v$) and their neighborhoods (generators $z_v$).  Since an automorphism of $\mathcal{G}$ is just a permutation of vertices that preserve neighborhoods,  we  deduce a first lemma. Explicitly:

\begin{lemma}\label{lemma: aut}
	Every $\sigma \in \Aut(\mathcal{G})$ induces $f_\sigma \in \Aut({M}_n(\mathcal{G}))$ defined by
	\[\begin{array}{lll}
	f_\sigma(\omega)=\omega, 		& \text{for $\omega \in \{x_1,x_2,y_1,y_2,y_3,z\}$,} \\
	f_\sigma(x_v) = x_{\sigma(v)}, 		& \text{for $v\in V(\G)$,} \\
	f_\sigma(z_v) = z_{\sigma(v)}, 		& \text{for $v\in V(\G)$.}
	\end{array}\]
Moreover, given $\sigma_1,\sigma_2 \in \Aut(\mathcal{G})$, then $f_{\sigma_1}\simeq f_{\sigma_2}$ if and only if $\sigma_1=\sigma_2$.
\end{lemma}
\begin{proof}
Given $\sigma \in \Aut(\mathcal{G})$, it is clear that $f_\sigma \in \Aut({M}_n(\mathcal{G}))$. Finally, if $\sigma_1,\sigma_2 \in \Aut(\mathcal{G})$, and $f_{\sigma_1}\simeq f_{\sigma_2}$, then for every $v\in V(\G)$ the elements $f_{\sigma_1}(x_v)=x_{\sigma_1(v)}$ and $f_{\sigma_2}(x_v)=x_{\sigma_2(v)}$ must by equal up to decomposable elements, that is $x_{\sigma_1(v)}=x_{\sigma_2(v)}$. Therefore $\sigma_1(v)=\sigma_2(v)$ for every $v\in V(\G)$, and $\sigma_1=\sigma_2$.
\end{proof}

For the second lemma, we argue as in \cite[Lemma 2.4]{CV2}:
\begin{lemma}\label{lemma: principal}
	For every $f\in\End\big({M}_n(\mathcal{G})\big)$, one of the following holds:
	\begin{enumerate}
		\item $f$ is an automorphism, and there exists $\sigma \in \Aut(\mathcal{G})$ such that
		\[\begin{array}{ll}
		f(\omega) = f_\sigma(\omega), 		& \text{for $\omega \in \big\{x_1,x_2,y_1,y_2,y_3,x_v\mid v\in V(\G)\big\}$,} \\
		f(z) = f_\sigma(z) + d(m_z),		& m_z \in\big({M}_n(\mathcal{G})\big)^{|z|-1}, \\
		f(z_v) = f_\sigma(z_v) + d(m_{z_v}), 	& \text{for $v\in V(\G)$ and }m_{z_v}\in\big({M}_n(\mathcal{G})\big)^{|z|-1},
		\end{array}\]
        where $f_\sigma$ is the morphism given by Lemma \ref{lemma: aut}.
		
        \item $\deg(f)=0$.
    \end{enumerate}
	In particular, we deduce that for $f\in \End \big({M}_n(\mathcal{G})\big)$, either $f \simeq f_\sigma, \, \sigma \in \Aut(\mathcal{G}),$ or $\deg(f)=0$.
\end{lemma}
\begin{proof}
We suppose that $\deg(f)\ne 0$, that is, $f$ does not fall into case (2). We have to show that $f$ falls into case (1).  	The element $f\in \End \big({M}_n(\mathcal{G})\big)$ is uniquely determined by the image of generators of ${M}_n(\mathcal{G})$.  	
First, we compute $f(x_1)$, $f(x_2)$, $f(y_1)$, $f(y_2)$, and $f(y_3)$.  For degree reasons (see Lemma \ref{lemma: isolated}) we have:
\begin{equation}\label{eq:f-parte-rigida}
    f(x_1) = a_1 x_1, \,
	f(x_2) = a_2 x_2, \,
	f(y_1) = b_1 y_1,\,
	f(y_2) = b_2 y_2,\,
	f(y_3) = b_3 y_3,
\end{equation}
for $a_i, b_j \in \Q$.
	Since $df(y_i)=f(dy_i)$ for $i=1,2,3$, we get that
	\begin{equation}\label{eq:relxiyi}
	b_1 = a_1^3 a_2, \quad  b_2 = a_1^2 a_2^2, \quad \text{and} \quad b_3 = a_1 a_2^3.
	\end{equation}
		
	Then, we continue with $f(x_v)$, for $v \in V (\mathcal G )$.  Observe that no pure power of $x_1$ can be a summand in $f(x_v)$. Indeed $|x_v| = |x_1|(6n+3)+(20n+12)$, with $20n+12 <  |x_1|$, for every $n\ge 1$, which means that $|x_1|$ does not divide $|x_v|$.   Secondly, no summand of the form $x_1^\alpha x_2^\beta$, with $\alpha$ and $\beta$ both positive can appear in $f(x_v) $, since $ |x_v|$ cannot be expressed as  $\alpha |x_1| + \beta  |x_2| $, with $\alpha, \beta$ both positive, by Lemma \ref{lem:dioph}. Then,
	\begin{equation}\label{eq:fxvinicial}
	f(x_v) = \sum_{w\in V (\mathcal G )} a(v,w) x_w + a_2(v) x_2^{5n+3} + A(v) y_1 y_2 + B(v) y_1 y_3 + C(v) y_2 y_3,
	\end{equation}
	where $a(v,w), a_2(v) \in \Q$ and $A(v), B(v), C(v) \in \mathbb{Q}[x_1,x_2]$, for all $v,w\in V(\G)$. Note that no term of the form $Ay_1y_2y_3$ or $Ay_i$, for $A\in\mathbb{Q}[x_1,x_2]$, can appear in the expression of $f (x_v)$ (see Remark \ref{rem:parity}).  Now, since $df(x_v)=f(dx_v)=0$, we get
	\[0 = A(v)(x_1^3 x_2 y_2 - x_1^2 x_2^2 y_1) + B(v)(x_1^3 x_2 y_3 - x_1 x_2^3 y_1) + C(v)(x_1^2 x_2^2 y_3 - x_1 x_2^3 y_2).\]
	In this equation the sum of the terms having $y_1$ as a factor, respectively $y_2$, $y_3$,  must all be zero. We start by paying attention to coefficients in $y_1$. As
	\[0 = A(v) x_1^2 x_2^2 + B(v) x_1 x_2^3 = x_1 x_2^2 (A(v) x_1 + B(v) x_2), \]
	we infer that $A(v)x_1 + B(v)x_2 = 0$, and therefore that $x_1$ must be a factor in $B(v)$ and $x_2$ a factor in $A(v)$. Therefore  $B(v)=\bar{B}(v)  x_1$  for some $\bar{B}(v)\in\mathbb{Q}[x_1,x_2]$, $A(v)=\bar{A}(v) x_2$ for some $\bar{A}(v)\in \mathbb{Q}[x_1,x_2]$ and $\bar{A}(v)x_1x_2 + \bar{B}(v) x_1x_2 = 0$, so that $\bar{A}(v)=-\bar{B}(v)$. We continue by paying attention to coefficients in $y_2$. As \[0 = A(v)x_1^3 x_2 - C(v) x_1 x_2^3 =  x_1 x_2 (A(v) x_1^2 - C(v)x_2^2),\]
	then $A(v) x_1^2 - C(v) x_2^2 = 0$.  But  $A(v)=\bar{A}(v)x_2$, so
	\[0 = \bar{A}(v) x_1^2 x_2 - C(v) x_2^2 = x_2(\bar{A}(v) x_1^2 - C(v) x_2) .\]
	The same kind of arguments as above,  imply now that $\bar{\bar{A}}(v)x_2 = \bar{A}(v)$ and $\bar{C}(v) x_1^2 = C(v)$, for some $\bar{\bar{A}}(v), \bar{C}(v)\in \mathbb{Q}[x_1,x_2]$.Then  $\bar{\bar{A}}(v) x_1^2 x_2 - \bar{C}(v) x_1^2 x_2 =0,$
	so $\bar{\bar{A}}(v)=\bar{C}(v)$.
Finally, we pay attention to coefficients in $y_3$. As
	\[0= B(v) x_1^3 x_2 + C(v) x_1^2 x_2^2 = x_1^2 x_2 (B(v) x_1 + C(v) x_2), \]
	we infer that $B(v) x_1 + C(v) x_2=0$. Since $B(v)=\bar{B}(v) x_1$ and $C(v) = \bar{C}(v) x_1^2$,
	\[0=\bar{B}(v) x_1^2 + \bar{C}(v) x_1^2 x_2 = x_1^2(\bar{B}(v) + \bar{C}(v) x_2),\]
	so $\bar{\bar{B}}(v) x_2 = \bar{B}(v)$, for some $\bar{\bar{B}}(v)\in\mathbb{Q}[x_1,x_2]$. Then $0=\bar{\bar{B}}(v) x_2 + \bar{C}(v)x_2$, and so $\bar{\bar{B}}(v)=-\bar{C}(v)$.
		
	Summarizing, we have
	$A(v) = \bar{C}(v) x_2^2$, $B(v) = -\bar{C}(v)x_1 x_2$ and $C(v)=\bar{C}(v) x_1^2$, hence \eqref{eq:fxvinicial} can be rewritten as:
	\[f(x_v) = \sum_{w\in V(\G)} a(v,w) x_w + a_2(v) x_2^{5n+3} + \bar{C}(v)(x_2^2 y_1 y_2 - x_1 x_2 y_1 y_3 + x_1^2 y_2 y_3).\]
	Let us now prove that $\bar{C}(v) = P(v) x_1 x_2$, with $P(v)\in\mathbb{Q}[x_1,x_2]$. Notice that if we succeed, as $d(y_1y_2y_3) = x_1^3 x_2 y_2 y_3 - x_1^2 x_2^2 y_1 y_3 + x_1 x_2^3 y_1 y_2$,  the last summand in $f (x_v)$ can be expressed as $d(P(v)y_1y_2y_3)$.   It suffices to prove that $|x_1|$ and $|x_2|$ are not divisors of the degree of $\bar{C}(v)$. This is clear as for $|\bar{C}(v)| = |x_v| - 2|x_2| - |y_1| - |y_2| = 180n^2 - 112n - 132 = |x_1|(6n-8) + (20n + 12)$, with $0 < 20n+12 < |x_1|$, for $n\ge 1$ and also for $ |\bar{C}(v)|= |x_2|(5n-7) + (30n + 22)$  with $0 < 30n+22 < |x_2|$ which leads us to the conclusion that
	$\bar{C}(v)$ cannot be a pure power of either $x_1$ or $x_2$. Then
	\begin{equation}\label{eq:fxvfinal}
	f(x_v) = \sum_{w\in V(\G)} a(v,w) x_w + a_2(v) x_2^{5n+3} + d\big(P(v)y_1 y_2 y_3\big).
	\end{equation}
		
{Third, we determine $f(z)$}. Remark that terms of the form $Py_1y_2y_3$ with $P\in \mathbb{Q}[x_1,x_2, x_v]$ cannot appear in the expression  of $f(z)$, since by Lemma \ref{lem:dioph} there is no polynomial $Q\in\mathbb{Q}[x_1,x_2]$ such that $|Qx_v^my_1y_2y_3|=|z|$, for $m = 0, 1, 2$. It is then clear that
\begin{equation}\begin{aligned}\label{eq:fzinicial}
	f(z) =&\ cz + \sum_{w\in V(\G)}  c(w)z_w  + A_1 y_1 + A_2 y_2 +  A_3 y_3 \\
	&+ \sum_{w\in V(\G)} x_w \big( B_1(w)  y_1 +  B_2(w) y_2 + B_3(w) y_3\big)\\
	&+\sum_{\{v,w\}\subset V(\G)} x_v x_w  \big(C_1(v,w)y_1 + C_2(v,w)y_2 + C_3(v,w)y_3\big),
\end{aligned}\end{equation}
	where $c, c(w) \in \Q$, $A_i$, $B_i(w)$, $C_i(v,w) \in \Q [x_1, x_2]$ for all $v,w\in V(\G)$, $i=1,2,3$.
	Since $df(z)$ and $f(dz)$ must coincide, on one hand using \eqref{eq:f-parte-rigida}:
\begin{equation}\begin{aligned}\label{eq:fdz}
	f(dz)  = &  b_1 b_2 a_1^{18n} a_2^2  y_1 y_2 x_1^{18n}x_2^2 - b_1 b_3 a_1^{18n+1} a_2 y_1y_3 x_1^{18n+1}x_2 \\
	& + b_2b_3 a_1^{18n+2}y_2y_3 x_1^{18n+2} + a_1^{18n+11} x_1^{18n+11} + a_2^{15n+9} x_2^{15n+9},
	\end{aligned}\end{equation}
	and on the other hand:
	\begin{equation}\begin{aligned}\label{eq:dfz}
	 df(z) =& c\big(x_1^{18n}(x_1^2y_2y_3 - x_1 x_2 y_1 y_3 + x_2^2 y_1 y_2) +  x_1^{18n+11} + x_2^{15n+9}\big) \\
	& + \sum_{w\in V(\G)} c(w) \left[ x_w^3 + \sum_{(w,u)\in E(\G)} x_w x_u x_2^{5n+3}\right]\\
    &+ A_1x_1^3 x_2 + A_2 x_1^2 x_2^2 + A_3 x_1 x_2^3 \\
	& + \sum_{w\in V(\G)} x_w\big[B_1(w) x_1^3 x_2 + B_2(w) x_1^2 x_2^2 + B_3(w) x_1 x_2^3\big]\\
	& + \sum_{\{u,w\}\subset V(\G)} x_u x_w \big[C_1(u,w) x_1^3 x_2 + C_2(u,w) x_1^2 x_2^2 + C_3(u,w) x_1 x_2^3\big].
	\end{aligned}\end{equation}
	Now notice that on \eqref{eq:fdz} there are no terms $x_1^\alpha x_2^\beta$ for $\alpha, \beta >0$, whereas in \eqref{eq:dfz} there are. This leads to $A_1x_1^3 x_2 + A_2 x_1^2 x_2^2 + A_3 x_1 x_2^3 = 0$ and thus to $d(\sum_{i=1}^3 A_i y_i)= 0$. Moreover, in \eqref{eq:fdz} there are no terms at all involving $x_w$, so  we obtain $c(w)=0$, $
B_1(u,w)x_1^3 x_2 + B_2(u,w) x_1^2 x_2^2 + B_3(u,w) x_1 x_2^3 = 0$, and $ C_1(u,w)x_1^3 x_2 + C_2(u,w) x_1^2 x_2^2 + C_3(u,w) x_1 x_2^3 = 0$
 for all $u,w\in V(\G)$ in \eqref{eq:dfz}. Hence $d\big(\sum_{i=1}^3  B_i(w) y_i\big)=0$ and $d\big(\sum_{i=1}^3 C_i(u,w)y_i\big)= 0$, and by Lemma \ref{prop:dmz}, we can finally reduce the expression to:
	\begin{equation} \label{eq:fzfinal}
	f(z) = cz + d(m_z), \quad m_z \in \big({M}_n(\mathcal{G})\big)^{|z|-1}.
	\end{equation}
	
	Before computing the last image $f(z_v)$, $v\in V(\G)$, we work out here the value of the remaining constants. Compare again \eqref{eq:fdz} to \eqref{eq:dfz} and  invoke \eqref{eq:relxiyi} to obtain:
	\[\left\{\begin{array}{l}
	c = a_1^{18n+2} b_2 b_3 = a_1^{18n+5}a_2^5, \\
	c = a_1^{18n+1} a_2 b_1 b_3 = a_1^{18n+5}a_2^5, \\
	c = a_1^{18n} a_2^2 b_1 b_2 = a_1^{18n+5} a_2^5, \\
	c = a_1^{18n+11}, \\
	c = a_2^{15n+9},
	\end{array}\right.\]
what is a particular case of Equation \eqref{eq: factorescaibi} for $k=6n+4$. Therefore either $a_1=a_2=b_1=b_2=b_3=c=0$ or $a_1=a_2=b_1=b_2=b_3=c=1$.
	Since $\deg(f)\ne 0$, we know that $f\in\E(M_n(\G))$ (Proposition \ref{prop:hopf}), and therefore $a_1\ne 0$, from which it follows that $a_1=a_2=b_1=b_2=b_3=c=1$.
	
	Finally, we end by computing $f(z_v)$, $v\in V(\G)$. Since $|z_v| = |z|$, the degree reasons used to describe $f(z)$  in \eqref{eq:fzinicial} apply, and $f(z_v)$ can be expressed as:
	\begin{equation}\label{eq:fzvinicial}\begin{aligned}
	f(z_v)= e(v)& z +  \sum_{w\in V(\G)} c(v,w) z_w + G_1(v) y_1 + G_2(v) y_2 + G_3(v) y_3 \\
	+ & \sum_{w\in V(\G)} x_w\big(H_1(v,w)y_1 + H_2(v,w) y_2 + H_3(v,w) y_3\big) \\
	+ & \sum_{\{u,w\}\subset V(\G)} x_u x_w \big(I_1(v,u,w)y_1 + I_2(v,u,w)y_2 + I_3(v,u,w)y_3\big),
	\end{aligned}\end{equation}
	with $e(v), c(v,w)\in \Q$, and $G_i(v), H_i(v,w), I_i(v,u,w) \in \Q[x_1,x_2]$, for all $v,u,w\in V(\G)$, $i=1,2,3$.
%$G_1(v)$, $G_2(v)$, $G_3(v)$, $H_1(v,w)$, $H_2(v,w)$, $H_3(v,w)$, $I_1(v,u,w)$, $I_2(v,u,w)$, $I_3(v,u,w) \in \Q[x_1,x_2]$, for all $v,u,w\in V(\G)$.

	Then,
\begin{equation}
\begin{aligned}\label{eq:dfzv}
	df(z_v)  =&\ e(v) \left[ x_1^{18n}(x_1^2 y_2 y_3 - x_1 x_2 y_1 y_3 + x_2^2 y_1 y_2) + x_1^{18n+11} + x_2^{15n+9} \right] \\
	& + \sum_{w\in V(\G)} c(v,w) \left( x_w^3 + \sum_{(w,u)\in E(\G)} x_w x_u x_2^{5n+3}\right) \\
	& + G_1(v) x_1^3 x_2 + G_2(v) x_1^2 x_2^2 + G_3(v) x_1 x_2^3 \\
	& + \sum_{w\in V(\G)} x_w \big(H_1(v,w)x_1^3 x_2 + H_2(v,w)x_1^2 x_2^2 + H_3(v,w) x_1 x_2^3\big) \\
	& + \sum_{\{w,u\}\subset V(\G)} x_w x_u \big(I_1(v,u,w)x_1^3 x_2 + I_2(v,u,w) x_1^2 x_2^2 + I_3(v,u,w) x_1 x_2^3\big).
	\end{aligned}
\end{equation}
	On the other hand, using \eqref{eq:fxvfinal}
	\begin{equation}\label{eq:fdzv}\begin{aligned}
	f(dz_v)  = \Biggl( & \sum_{w\in V(\G)}  a(v,w) x_w + a_2(v) x_2^{5n+3} + d\big(P(v)y_1y_2y_3\big) \Biggr)^3 \\
	+ \sum_{(v,r)\in E(\G)} &\left(\sum_{w\in V(\G)} a(v,w) x_w + a_2(v) x_2^{5n+3} + d\big(P(v)y_1y_2y_3\big) \right) \\
	& \times \left( \sum_{u\in V(\G)} a(r,u)x_u + a_2(r) x_2^{5n+3} + d\big(P(r)y_1y_2y_3\big)\right) x_2^{5n+3}.
	\end{aligned}\end{equation}
First, notice that there is no summand in \eqref{eq:dfzv} of the form $x_v x_w x_u$, where $u$, $v$ and $w$ are pairwise distinc. This forces \eqref{eq:fdzv} to have at most two nontrivial coefficients $a(v,w)$. Furthermore, in \eqref{eq:dfzv} there is no factor of the form $x_w^2 x_v$, where $w\ne v$, which forces exactly one non trivial coefficient $a(v,w)$   in \eqref{eq:fdzv} (recall that $f\in\E(M_n(\G))$, hence at least one coefficient $a(v,w)$ must be nontrivial). Consequently,  in \eqref{eq:fdzv} there is a unique summand containing $x_w^3$ and, in \eqref{eq:dfzv} there can only be one non trivial coefficient $c(v,w)$.  This fact can be read as there is a self-map of $V(\G)$, $\sigma$, such that $a\big(v,\sigma(v)\big)$ and $c\big(v,\sigma(v)\big)$ are the only non trivial coefficients.
	Second, notice that in \eqref{eq:fdzv} there is no factor that is a pure power of $x_1$, and thus, $e(v)=0$ in \eqref{eq:dfzv}. Also, since the graph has no loops, a term of type $x_{\sigma(v)}^2 x_2^{5n+3}$ does not appear in \eqref{eq:dfzv}, and $a\big(v,\sigma(v)\big)\ne 0$ implies that $a_2(v) = 0$. We focus now on \eqref{eq:fdzv}. Notice that
	\[d\big(P(v)y_1y_2y_3\big) = P(v)(x_1^3 x_2 y_2 y_3 - x_1^2 x_2^2 y_1 y_3 + x_1 x_2^3 y_1 y_2),\]
	and if two terms of this form are multiplied, their product is zero ($y_i^2 =0$). Thus,
\begin{equation}
\begin{aligned}\label{eq:dfzv_2}
	f(dz_v) = &\ a\big(v,\sigma(v)\big)^3 x_{\sigma(v)}^3 + 3 a\big(v,\sigma(v)\big)^2 x_{\sigma(v)}^2 d\big(P(v)y_1 y_2 y_3\big) \\
	+ &\ \sum_{(v,r)\in E(\G)} \biggl[ a\big(v,\sigma(v)\big) a\big(r,\sigma(r)\big) x_{\sigma(v)} x_{\sigma(r)}\\
	+ &\  a\big(v,\sigma(v)\big) x_{\sigma(v)} d\big(P(r) y_1 y_2 y_3\big) \\
    + &\ a\big(r,\sigma(r)\big) x_{\sigma(r)} d\big(P(v) y_1 y_2 y_3\big) \biggr] x_2^{5n+3}.
	\end{aligned}
\end{equation}
	Observe that only one term containing $x_{\sigma(v)}^2$ appears in \eqref{eq:dfzv_2}, which is multiplied by $d\big(P(v)y_1 y_2 y_3\big) = P(v)(x_1^3 x_2 y_2 y_3 - x_1^2 x_2^2 y_1 y_3 + x_1 x_2^3 y_1 y_2)$. However, in \eqref{eq:dfzv} there is no term containing both $x_{\sigma(v)}^2$ and $y_i$.  Therefore, we can conclude that $P(v)=0$, for every $v\in V(\G)$, and hence
\begin{equation}\label{eq:dfzv_3}
f(dz_v) = a\big(v,\sigma(v)\big)^3 x_{\sigma(v)}^3 + \sum_{(v,r)\in E(\G)} a\big(v,\sigma(v)\big) a\big(r,\sigma(r)\big) x_{\sigma(v)} x_{\sigma(r)} x_2^{5n+3}.
\end{equation}
	Now, comparing  terms containing $x_1 x_2$, $x_1 x_2 x_w$ and $x_1 x_2 x_w x_u$  in \eqref{eq:dfzv} and \eqref{eq:dfzv_3}, we get:
	\[
	\left\{\begin{array}{l}
	G_1(v) x_1^3 x_2 + G_2(v) x_1^2 x_2^2 + G_3(v) x_1 x_2^3 = 0 ,\\
	H_1(v,w)x_1^3 x_2 + H_2(v,w)x_1^2 x_2^2 + H_3(v,w) x_1 x_2^3 = 0, \\
	I_1(v,u,w)x_1^3 x_2 + I_2(v,u,w) x_1^2 x_2^2 + I_3(v,u,w) x_1 x_2^3 = 0.
	\end{array}\right.
	\]
	In other words $\sum_{i=1}^3 y_iG_i(v)$, $\sum_{i=1}^3 y_iH_i(v,w) $ and $\sum_{i=1}^3 y_i I_i(v,u,w) $ are cocycles in degrees $|z|$, $|z|-|x_v|$ and $|z|-2|x_v|$ respectively. We now invoke Lemma \ref{prop:dmz} to conclude that they are also coboundaries.  Then, we can finally write
	$$f(z_v) = c(v,\sigma(v)) z_{\sigma(v)} + d(m_{z_v}),$$ for some $m_{z_v}\in \big(M_n(\G)\big)^{|z|-1}$,
	and
\begin{equation}\label{eq:dfzv_4}
df(z_v)  =  c(v,\sigma(v)) \left( x_{\sigma(v)}^3 + \sum_{(\sigma(v),u)\in E(\G)} x_{\sigma(v)} x_u x_2^{5n+3}\right).
\end{equation}
	Comparing both \eqref{eq:dfzv_3} and \eqref{eq:dfzv_4}, we get
		\[\begin{array}{ll}
	c\big(v,\sigma(v)\big) = a\big(v,\sigma(v)\big)^3, & \text{for all $v\in V(\G)$,} \\
	c\big(v,\sigma(v)\big) = a\big(v,\sigma(v)\big) a\big(w,\sigma(w)\big), & \text{for all $(v,w)\in E(\G)$.}
	\end{array}\]
	Therefore $a\big(v,\sigma(v)\big)^2 = a\big(w,\sigma(w)\big)$, if $(v,w)\in E(\G)$. Observe that since $\mathcal{G}$ is not a directed graph, $(w,v)$ is also an edge of $\mathcal G$, then  $a\big(w,\sigma(w)\big)^2 = a\big(v,\sigma(v)\big)$ and $a\big(v,\sigma(v)\big)^4 = a\big(v,\sigma(v)\big)$. Moreover, since $\mathcal{G}$ is connected, and $a\big(v,\sigma(v)\big)\ne 0$, then $a\big(v,\sigma(v)\big) = c\big(v,\sigma(v)\big) = 1$ for all $v\in V(\G)$.  Therefore we can write:
		\begin{equation*}\begin{aligned}
		f(x_v) &= x_{\sigma(v)}, \\
		f(z_v) &= z_{\sigma(v)} + d(m_{z_v}).
		\end{aligned}\end{equation*}
		We shall see that in this case $\sigma\colon V(\G)\to V(\G)$ is, in fact, an element in $\Aut(\mathcal{G})$. We first show that $\sigma$ is a full endomorphism of the graph $\G$, that is, that $(v, w)\in E(\G)$ if and only if $\big(\sigma(v), \sigma(w)\big)\in E(\G)$. Indeed, $(v, w)\in E(\G)$ if and only if there is a summand $x_v x_w x_2^{5n+3}$ in $dz_v$,  and this is equivalent to there being a summand $x_{\sigma(v)} x_{\sigma(w)} x_2^{5n+3}$ in $f(dz_v)=df(z_v)=d(z_{\sigma(v)})$,  or equivalently, $\big(\sigma(v),\sigma(w)\big)\in E(\G)$. Now, since for every $v\in V(\G)$, $f(dz_v)=d(z_{\sigma(v)})$, $\sigma$ is one-to-one on the neighborhood of every vertex. Thus $\sigma \in \Aut(\mathcal{G})$, which finally proves that if  $\deg(f)\ne 0$, then  $f$ falls into case (1) of Lemma \ref{lemma: principal}.
\end{proof}

Gathering  Lemmas \ref{lemma: aut}, and \ref{lemma: principal},  we have
succeeded in proving the following:
\begin{lemma}\label{th:sullgraph}
 The monoid of homotopy classes of self-maps of $M_n(\G)$ is:
$$[M_n(\G),M_n(\G)] \cong \Aut (\G) \sqcup \{f \colon\,\deg(f)=0\} .$$
\end{lemma}

As a consequence, we obtain our main result, Theorem \ref{teor:sullgraph}.

\emph{Proof of Theorem \ref{teor:sullgraph}:} From Proposition \ref{prop:hopf}, we know that self-maps $f$ with $\deg(f)=0$ are not self-homotopy equivalences. Then, using Lemma \ref{lemma: aut}, we conclude that $\Aut (\G)\cong\E (M_n(\G))$, for every $n>1$.

\section{Inflexible manifolds} \label{sec: inflex}

A closed, oriented and connected manifold $M$ is said to be inflexible if the set of all the possible degrees of its continuous self-maps is finite. As the degree is multiplicative, this condition is equivalent to ask for the set of all the possible degrees to be a subset of $\{-1, 0, 1\}$.
Inflexible manifolds naturally appear within the framework of functorial seminorms on singular homology developed by Gromov \cite{Gromov, Gromov2} and derived degree theorems (e.g.\ \cite[Remark 2.6]{CL}): let $M$ be a closed, oriented and connected manifold with fundamental class $c_M$. If there exists a functorial seminorm on singular homology $\vert\cdot\vert$ such that $\vert c_M\vert\ne 0$, then $M$ is inflexible. In this way, oriented closed connected hyperbolic manifolds are shown to be inflexible; they do have non trivial simplicial volume, the value of ${\ell}^1$-seminorm applied to the fundamental class \cite[Section 0.3]{Gromov2}. But ${\ell}^1$-seminorm is trivial on simply connected manifolds \cite[Section 3.1]{Gromov2}, which led Gromov to raise the question of whether
every functorial seminorm on singular
homology is trivial on all simply connected spaces \cite[Remark (b) in 5.35]{Gromov}. This question is solved in the negative in \cite{CL} by constructing functorial seminorms associated to simply connected inflexible manifolds. Therefore, inflexible manifolds are extraordinary objects and still not many examples are known. Indeed, all the examples found in literature show low levels of connectivity when observing their minimal Sullivan models \cite{AL2,Am,CL,CV2}.

In this section, we closely follow the lines of \cite{CV2} and provide new examples of inflexible manifolds whose Sullivan models are as highly connected as desired (see Corollary \ref{cor:inflexman}). To that purpose, we first need to recall some tools.
	
 First, we recall that for elliptic Sullivan algebras, the Barge and Sullivan obstruction theory (\cite{Ba,Su}) decides if there exists a manifold (over $\mathbb Z$) of the same rational homotopy type of that algebra. Roughly speaking, the obstruction theory is trivial when the formal dimension of the algebra is not congruent to $0\pmod{4}$. Our examples will fall into this case.

Secondly, we recall a construction from  \cite[Proposition 3.1, Lemma 3.2]{CV2}.  Let  $A$ be a $1$-connected elliptic Sullivan algebra of formal dimension $2m$. Then,
 we can construct a $1$-connected elliptic Sullivan algebra  of formal dimension $4m-1$ as follows: we choose a representative of the fundamental class of $A$, let us say $x$ with $|x| = 2m$ and we define $\tilde A = (A \otimes \Lambda (z), d z = x)$.  The algebra $\tilde A$ inherits properties of $A$: it is elliptic, the monoid of self-maps of $\tilde A$ coincides with the monoid of self-maps of $A$, and finally, the connectivity is preserved.

At this point, we wish to explain our strategy. We have proved that ${M}_n(\mathcal{G})$ from Definition \ref{def: m_nG}  is an elliptic Sullivan algebra of even dimension,  let us say $2m$, (see Lemma \ref{lemma: ellipticminimal}). By means of the construction above, we define a new elliptic Sullivan algebra $\tilde {M}_n(\mathcal{G})$ of formal dimension $4m-1$.  The monoid of self-maps of  both algebras  coincide, hence $\E\big(\tilde {M}_n(\mathcal{G})\big)$  is finite which implies that $\tilde {M}_n(\mathcal{G})$ is inflexible.  Also $\tilde {M}_n(\mathcal{G})$ and ${M}_n(\mathcal{G})$ have the same connectivity. Now, as we have already mentioned,  the obstruction of Barge and Sullivan is trivial  and $\tilde {M}_n(\mathcal{G})$ has the rational homotopy type of a manifold.  Even more,  by Proposition \ref{PascalDon}, $\tilde {M}_n(\mathcal{G})$ has the rational homotopy type of $\M_n(\G)$,  a closed smooth manifold with the same connectivity as $\tilde {M}_n(\mathcal{G})$.

We can now prove the main theorem in this section.
\begin{theorem}\label{thminflexman}
For any finite group $G$ and any integer $n \geq 1$, $G$ is the group of self-homotopy equivalences of the rationalization of an inflexible manifold which is  $(30n+17)$-connected.
\end{theorem}

\begin{proof}
Take $\G$ a graph such that $\Aut(\G)\cong G$ \cite{Frucht1}.  For any $n\ge 1$, consider the algebra $M_n(\G)$ from Definition \ref{def: m_nG}. We know from Theorem \ref{teor:sullgraph} that $\E\big(M_n(\G)\big)\cong \Aut(\G)$. We also know from Lemma \ref{lemma: ellipticminimal} that $M_n(\G)$ is an elliptic minimal Sullivan algebra of formal dimension an even integer $2\delta=(18n+22)\deg x_1 + 2 |V(\G)| \deg x_v$. We apply \cite[Proposition 3.1]{CV2} to obtain an elliptic Sullivan algebra of formal dimension $\tilde{\delta}=4\delta-1$, which we denote $\tilde M_n(\G)$, such that $\E\big(\tilde M_n(\G)\big) \cong \E\big(M_n(\G)\big)\cong G$,  by Lemma \ref{th:sullgraph}.
Since $\tilde M_n(\G)$ is elliptic, every non zero degree map must be a self-homotopy equivalence by Proposition \ref{prop:hopf}, and since $\E\big(\tilde M_n(\G)\big)$ is finite, the set of degrees is finite as well and $\tilde M_n(\G)$ is an inflexible Sullivan algebra.
Now $\tilde{\delta}\not\equiv 0\pmod{4}$,  so by Proposition \ref{PascalDon}, $\tilde{M}_n(\G)$ is the rationalization of a  $(30n+17)$-connected manifold $\M_n(\G)$. Observe that $\M_n(\G)$ is also inflexible since so is $\tilde M_n(\G)$.  Thus
\[G \cong \Aut(\G) \cong \E\big(M_n(\G)\big) \cong \E\big(\tilde M_n(\G)\big) \cong \E\big(\M_n(\G)_\Q\big),\]
where $\M_n(\G)_\Q$ is the rational homotopy type of $\M_n(\G)$.

\end{proof}
In particular, we immediately obtain the following.
\begin{corollary}\label{cor:inflexman}
There exist infinitely many non-homotopically equivalent inflexible manifolds as highly connected as desired.
\end{corollary}

We end with some results on the existence of strongly chiral manifolds, that is, manifolds that do not admit orientation-reversing self-maps of degree $-1$ (see \cite{Pup}, \cite{Am} or \cite{CV2}).
\begin{corollary}\label{cor:chiral}
For any $n\ge 1$ and $k>1$, there exists a rationally $(30n+17)$-connected strongly chiral manifold $M$ of dimension $(1080+720k)n^2+(1968+872k)n+264k+791$.
\end{corollary}

\begin{proof}
Let $\G$ be a connected simple graph with $k$ vertices, and let $\M_n(\G)$ be the corresponding manifold obtained in the proof of Theorem \ref{thminflexman} above, its minimal Sullivan algebra being $\tilde M_n(\G)$. Any self-map of $\tilde M_n (\G)$ is shown in \cite[Lemma 3.2]{CV2} to verify that $\deg (\widetilde f) = \deg (\widetilde f \vert_{M_n(\G)})^2$, thus it has degree $0$ or $1$. \end{proof}

Manifolds provided by Corollary \ref{cor:chiral} can be used to construct inflexible and strongly chiral product manifolds by exploiting techniques from \cite{Neo}. In \cite[Example 3.7]{Neo}, it is shown that given a simply connected inflexible manifold $N$ one can construct an inflexible product manifold $M\times N$ just by considering $M$ a closed oriented inflexible non-simply connected manifold that
does not admit maps of non-zero degree from direct products and $\dim N <\dim M$. Here we prove an ``inverse'':

\begin{corollary}\label{cor:products}
Let $M$ be a non necessarily simply connected closed oriented inflexible (resp.\ strongly chiral) manifold that
does not admit maps of non-zero degree from direct products. Then there exist simply connected strongly chiral manifolds $N$ such that $\dim N >\dim M$ and $M\times N$ is inflexible (resp.\ strongly chiral).
\end{corollary}

\begin{proof}
Let $m=\operatorname{dim} M$, and $n\ge 1$ be an integer such that $m < 30n+17$. Let $N$ be any of the $(30n+17)$-connected strongly chiral manifolds from Corollary \ref{cor:chiral}. Since $H^m(N;\Q)=0$, any continuous map $f\colon N\to M$ maps $H^m(M;\Q)$ to $0$. Thus the pair $(M,N)$ is under the assumptions of \cite[Theorem 1.4]{Neo}, and the result follows from \cite[Corollary 1.5(c)]{Neo} (resp.\ \cite[Corollary 1.5(a)]{Neo}).
\end{proof}

\appendix
\section{By Pascal Lambrechts and Don Stanley}
\subsection*{Realization of rationally Poincar\'e duality complexes by highly connected bounding manifolds}

\begin{proposition}\label{PascalDon}
Let $X$ be a $k$-connected rational space ($k\geq1$) whose rational cohomology algebra satisfies
Poincar\'e duality in dimension $n\geq5$. When $n$ is a multiple of $4$, suppose furthermore that the quadratic form on
$H^{n/2}(X;\Q)$ has signature $0$ and descends from a quadratic form over $\Z$. Then there exists a $k$-connected
closed smooth  manifold rationally equivalent
to $X$ and that is a boundary of
a compact oriented manifold.
\end{proposition}
\begin{proof}
Let $\varepsilon$ be some trivial vector bundle of rank $N>n$ over $X$.
Then the Thom space of $\varepsilon$, $\operatorname{Th}(\varepsilon)$, is homotopy equivalent to $\Sigma^N(X_+)$ which is a wedge of rational spheres.
Since $X$ has Poincar\'e duality in dimension $n$ there is a homotopy class
$\alpha\in\pi_{n+N}(\operatorname{Th}(\varepsilon))$  whose image under the Hurewicz map is non trivial.
One can pick a representative $f\colon S^{n+N}\to \operatorname{Th}(\varepsilon)$ of $\alpha$ transverse to $X$
and then the preimage $M':=f^{-1}(X)\subset S^{n+N}$ is a smooth closed $n$-manifold with trivial stable normal bundle
which comes with a map $f\colon M'\to X$ non trivial in $H^n(-;\Q)$. Since $X$ is $k$-connected, the
 classical surgery argument makes $M'$ $k$-connected.
Further surgeries of dimensions $\geq k$
can turn $f$ into a rational homology equivalence up to the middle dimension. As explained in \cite[Section 13]{Su}
or \cite{Ba} the obstruction to do surgery in the middle dimension in our case vanish if the dimension is not a
multiple of $4$ or if the dimension is a multiple of $4$ and the quadratic form on $H^{n/2}(X;\Q)$ is equivalent to
$\sum_{i=1}^l(x_i^2-y_i^2)$ (that is it descends from a quadratic form over $\Z$ and the signature vanishes). This gives a $k$-connected manifold $M$ with a rational homotopy equivalence $M\simeq_\Q
X$.
Furthermore all the Pontrjagin and Stiefel-Whitney numbers of $M'$ vanish because it  has a trivial stable normal bundle,
which implies that $M'$, and hence $M$, is the boundary of an oriented compact manifold.
\end{proof}
\begin{remark} Jim Fowler and Zhixu Su have proved that we cannot drop the hypothesis that the signature vanishes when the dimension is a multiple of $4$, even if
 $X$ is realizable by a smooth manifold. Indeed, consider the unique simply-connected rational
  homotopy type $X$  with $H^*(X;\Q)\cong\Q[x]/(x^3)$ with $\deg(x)=16$. This rational homotopy type is realizable by a
  closed smooth manifold but it cannot be $2$-connected
(\cite[Theorem B]{FoSu:smo}).
\end{remark}
%\bibliographystyle{plain}
%\bibliography{/Users/pascal/Library/texmf/tex/latex/bibliographies/PLbiblio.bib}
%\def\cprime{$'$}
%\begin{thebibliography}{1}

\bibliographystyle{amsalpha}

\end{document}